\documentclass{article}
\usepackage{graphicx}
\usepackage{comment}
\usepackage{amsmath}
\usepackage{amsthm}
\usepackage{amssymb}
\usepackage[colorlinks=true,linkcolor=blue,citecolor=green]{hyperref}
\usepackage[margin=3cm]{geometry}
\usepackage{stmaryrd}
\usepackage{bussproofs}
\usepackage{mathtools}
\usepackage[all]{xy}
\usepackage{bussproofs}
\usepackage{mathrsfs}
\usepackage{forest}
\forestset{smullyan tableaux/.style={for tree={math content},where n children=1{!1.before computing xy={l=\baselineskip},!1.no edge}{},closed/.style={label=below:$\times$},},}
\newcommand{\Luk}{{\mathchoice{\mbox{\sf\L}}{\mbox{\sf\L}}{\mbox{\sf\scriptsize\L}}{\mbox{\sf\tiny\L}}}}
\newcommand{\GLuk}{\mathcal{G}\Luk\mathsf{uk}}
\newcommand{\HLuk}{\mathcal{H}\Luk\mathsf{uk}}
\newtheorem{theorem}{Theorem}[section]

\newtheorem{corollary}{Corollary}[section]
\newtheorem{lemma}{Lemma}[section]
\newtheorem{case}{Case}[theorem]

\theoremstyle{definition}
\newtheorem{definition}{Definition}[section]
\theoremstyle{remark}

\newtheorem{convention}{Convention}[section]

\usepackage{longtable}
\usepackage{fancyhdr}
\usepackage{authblk}
\usepackage{soul}
\begin{document}
\providecommand{\keywords}[1]
{
  \small	
  \textbf{\textit{Keywords: }} #1
}
\title{Generalisation of proof simulation procedures for Frege systems by M.L.~Bonet and S.R.~Buss\thanks{The author wishes to thank V.~Shangin for his appreciative help and fruitful advice and two anonymous reviewers for their helpful comments which improved the quality of the paper.\\This is a~postprint version of the following paper --- \href{https://doi.org/10.1080/11663081.2018.1525208}{doi: 10.1080/11663081.2018.1525208}.}}
\author{Daniil Kozhemiachenko}
\affil{Department of Logic, Faculty of Philosophy, Lomonosov Moscow State University}
\date{}
\maketitle
\begin{abstract}
In this paper we present a~generalisation of proof simulation procedures for Frege systems by Bonet and Buss to some logics for which the deduction theorem does not hold. In particular, we study the case of finite-valued \L{}ukasiewicz logics. To this end, we provide proof systems $\Luk_{3_{n\vee}}$ and $\Luk_{3_\vee}$ which augment Avron's Frege system $\HLuk$ with nested and general versions of the disjunction elimination rule, respectively. For these systems we provide upper bounds on speed-ups w.r.t.\ both the number of steps in proofs and the length of proofs. We also consider Tamminga's natural deduction and Avron's hypersequent calculus $\GLuk$ for 3-valued \L{}ukasiewicz logic $\Luk_3$ and generalise our results considering the disjunction elimination rule to all finite-valued \L{}ukasiewicz logics.

\end{abstract}
\section{Introduction}
Proof simulation is a~well-established field in proof theory and complexity theory. One of the most known and important results is due to Cook and Reckhow~\cite{CookReckhow1974,Reckhow1976,CookReckhow1979} that Frege systems\footnote{Also called Hilbert-style calculi or axiomatic calculi. In this paper we will use the term \textit{Frege system}.}, natural deduction and sequent calculi with cut for classical logic polynomially simulate each other and that $\mathcal{NP}=co\mathcal{NP}$ if and only if there is a~proof system for classical logic that proves any tautology in a~polynomial time.

According to D'Agostino~\cite{DAgostino1990}, the work in proof simulation is conducted as follows. Given two proof systems it is either established that they polynomially simulate (cf,~Definition~\ref{simulation} for the notion of simulation) one another, that neither of them simulates another one, or that only one of the two systems polynomially simulates another one. When polynomial simulation is possible, both the simulation procedure and the resulting upper bound are provided.

When a~simulation procedure is provided, there are two main ways to compare proofs in different calculi: by the number of symbols (i.e., \textit{length} or \textit{size}) and the \textit{number of steps}. The second measure is usually easier to obtain but it can sometimes be insufficient since there are some instances of proof translations which provide exponential increase in size but only polynomial increase in the number of steps (cf.~\cite{Buss1987} for more details). These results illustrate a~possible discrepancy between two ways of comparing complexity of different calculi. On the other hand, Frege systems simulate each other polynomially~\cite{Reckhow1976,CookReckhow1979} and are ‘robust’ in the sense that they simulate each other polynomially not only w.r.t.\ the number of steps but also w.r.t.\ size~\cite{PudlakBuss1995}.

Differentiating between degrees of polynomials when providing upper bounds on proof speed-ups allows us to understand, among other things, the relative efficiency of different calculi. This is done, for example, in~\cite{BonetBuss1993}, where Bonet and Buss introduce several versions of the deduction rule: \[\dfrac{\Gamma,A\vdash B}{\Gamma\vdash A\supset B}\]

\textit{Simple deduction rule} which requires the elimination of all assumptions at once.
\[\dfrac{A_1,\ldots,A_m\vdash B}{\vdash A_1\supset(\ldots\supset(A_m\supset B))}\]

\textit{Nested deduction rule} which requires the elimination of the assumptions in the reverse order to their introduction.
\begin{longtable}{l}
$\left[\begin{tabular}{l}
$A$ --- assumption\\
\vdots\\
$B$
\end{tabular}\right.$\\
$A\supset B$ --- deduction rule\\
\end{longtable}

\textit{General deduction rule} that allows for the elimination of any assumption.
\[\dfrac{\Gamma\vDash B}{\Gamma\setminus\{A\}\vDash A\supset B}\]

Bonet and Buss provide simulation procedures and upper bounds for speed-ups w.r.t.\ the number of steps of Frege systems equipped with one of these rules over Frege systems without any additional rules. They apply these results to obtain upper bounds on speed-ups of natural deduction and Gentzen-style sequent calculus over Frege systems as well as of Frege systems with dag-like proofs over Frege systems with tree-like proofs.

The proofs provided by Bonet and Buss have three characteristic features. First, they deal with classical logic only (although their results also hold for some non-classical logics as shown in~\cite{BolotovKozhemiachenkoShangin2018}). Second, they rely heavily on the deduction theorem. Third, they extensively use the fact that there is a~formula containing only $A$ and $B$ such that the formula is valid iff $A$ classically entails $B$. This third feature is what allows Bonet and Buss to simulate various calculi with Frege systems.

On the other hand, there is a~well-known generalisation of the result by Cook and Reckhow: $\mathcal{NP}=co\mathcal{NP}$ \textit{iff there is a~polynomially bounded proof system for a~logic with $co\mathcal{NP}$-complete set of tautologies}\footnote{The author would like to thank J.~Kraj\'{i}\v{c}ek who drew his attention to this fact in personal correspondence.}. This fact implies that we can study simulations of proof systems for non-classical logics whose tautologies are $co\mathcal{NP}$-complete for the same reason we do it in classical logic. Moreover, since non-classical logics have fewer tautologies and admit fewer rules than classical logic, it may be easier to find the tautologies that separate proof systems from each other and establish their lower bounds. It is also possible that there would be some non-classical tautologies which are hard for their proof systems but provable in a~polynomial time in classical logic just as in the case of intuitionism (see~\cite{Hrubes2007,Hrubes2009} for more details).

In his paper~\cite{Hahnle2003} H\"{a}hnle provides the following list of non-classical logics whose sets of valid formulas are $co\mathcal{NP}$-complete: both finite- and infinite-valued G\"{o}del's logics and \L{}ukasiewicz logics and infinite-valued product logics. In this paper, we will study proof simulations of finite-valued \L{}ukasiewicz logics for the following reasons.

First, finite-valued \L{}ukasiewicz logics do not (in contrast to G\"{o}del's logics) have the deduction theorem since $A\vDash B$ does not entail $\vDash A\supset B$. Hence, there is no ‘natural deduction’ for them in the sense of Reckhow~\cite{Reckhow1976} and Pelletier~\cite{Pelletier1999}. However, it is possible to construct an ND-like calculus for $\Luk_3$ which will allow opening and closing arbitrary assumptions in the proofs of theorems (cf.~\cite{Tamminga2014} for the definition of such a~calculus and~\cite{PetrukhinShangin2018} for a~proof searching algorithm for it) since reasoning by cases (the disjunction elimination rule) is valid for all finite-valued \L{}ukasiewicz logics. For simplicity's sake, we will further call the calculus by Tamminga from~\cite{Tamminga2014} ‘natural deduction for $\Luk_3$’. There is also a~hypersequent calculus by Avron~\cite{Avron1991} for $\Luk_3$. That is why our primary concern in this paper will be 3-valued \L{}ukasiewicz logic, although we will generalise some results to all finite-valued \L{}ukasiewicz logics.

The second reason is that all finite-valued \L{}ukasiewicz logics have formulas containing only $A$ and $B$ that are valid iff $A$ entails $B$ for arbitrary $A$ and $B$. Namely, we have
\[A\vDash_{\Luk_k}B\text{ iff }\vDash_{\Luk_k}\underbrace{A\supset(\ldots\supset(A}_{k-1\text{ times}}\supset B))\text{ iff }\vDash_{\Luk_k}\neg\underbrace{(A\supset(\ldots\supset(A\supset\neg A)))}_{k-1\text{ times}}\supset B\]
where $\vDash_{\Luk_k}$ is the entailment relation of $k$-valued \L{}ukasiewicz logic. This fact discerns them from infinite-valued \L{}ukasiewicz logic and product logic for which only a~weaker statement holds: \textit{$A\vDash B$ iff there exists $n$ such that $\underbrace{A\supset(\ldots\supset(A}_{n\text{ times}}\supset B))$ is valid}~\cite{Hajek1998}.

It should also be noted that the simulation procedures described by Buss and Bonet in~\cite{BonetBuss1993} cannot be applied to the simulation of natural deduction for \L{}ukasiewicz logic $\Luk_3$ or calculi obtained by adding the disjunction elimination rule because the said procedures use deduction theorem which does not hold in $\Luk_3$. On the other hand, a~simulation procedure for $\Luk_3$ will hold for the classical case since all $\Luk_3$-valid formulas and all $\Luk_3$-entailments are also classically valid. We will thus be able to generalise simulation procedures devised by Buss and Bonet to logics without the deduction theorem and show that it is still possible to devise a~version of the disjunction elimination rule that will have the same upper bound on speed-up over Frege systems as in classical logic.

These results will enable us to consider upper bounds on speed-ups of natural deduction over Frege systems for $\Luk_3$. Moreover, we will generalise our simulation procedures for every finite-valued \L{}ukasiewicz logic in such a~way that upper bounds on speed-ups will be the same for all of them. Furthermore, we will establish a~polynomial simulation of Avron's hypersequent calculus $\GLuk$ by the $\HLuk$. All upper bounds on speed-ups are given w.r.t.\ both the \textit{number of steps} and the \textit{length of proofs} except for simulation of $\GLuk$, where we consider simulation w.r.t.\ the length of proofs only. Although these results do not seem surprising, once we take into account what we already know about proof simulations, there are, as far as we know, no results considering proof simulation and upper bounds on speed-ups of different proof systems for finite-valued \L{}ukasiewicz logics.

The structure of our paper is as follows. To make our paper self-contained, we define the semantics of 3-valued \L{}ukasiewicz logic and describe proof systems $\Luk_{3_{n\vee}}$ and $\Luk_{3_\vee}$ obtained from the Frege system $\HLuk$ of Avron~\cite{Avron1991} in  section~\ref{Luk3}. Then in section~\ref{mainresults} we provide upper bounds on speed-ups of $\Luk_{3_{n\vee}}$ and $\Luk_{3_\vee}$ over $\HLuk$. In section~\ref{NDandGLuk} we consider upper bounds on speed-ups of Avron's hypersequent calculus $\GLuk$ and natural deduction by Tamminga over $\HLuk$ and in section~\ref{Lk} we generalise results considering $\Luk_{3_{n\vee}}$ and $\Luk_{3_\vee}$ to the case of an arbitrary finite valued \L{}ukasiewicz logic $\Luk_k$\footnote{For all proof systems except for $\GLuk$ we provide upper bounds on speed-ups w.r.t.\ the \textit{number of steps}. For the hypersequent calculus we provide its upper bound on speed-up over $\HLuk$ w.r.t.\ the \textit{length of proof}.}. Finally, in section~\ref{conclusion} we sum up our work and provide a~roadmap for future research.
\section{3-valued \L{}ukasiewicz logic: semantics and proof systems\label{Luk3}}
We will use propositional fragment of 3-valued \L{}ukasiewicz logic over $\wedge,\vee,\supset,\neg$.
\begin{definition}\label{L3definition}
The syntax of $\Luk_3$ is described using Backus--Naur notation as follows.
\[A\coloneqq p_i\mid(A\wedge A)\mid(A\vee A)\mid(A\supset A)\mid\neg A\]

The logic has the following semantics.
\begin{enumerate}
\item A valuation $v$ maps a~set of propositional variables to $\left\{1,\frac{1}{2},0\right\}$.
\item $v(\neg A)=1-v(A)$.
\item $v(A\wedge B)=\min(v(A),v(B))$.
\item $v(A\vee B)=\max(v(A),v(B))$.
\item $v(A\supset B)=\min(1,1-v(A)+v(B))$.
\end{enumerate}

Finally, we define the notions of validity and entailment.
\begin{enumerate}
\item A formula $A$ is valid (we will write $\vDash A$ in this case) iff it takes $1$ under any valuation.
\item  $B$ follows from $A$ or, equivalently, $A$ entails $B$ ($A\vDash B$) iff for any valuation $v$ such that $v(A)=1$, it holds that $v(B)=1$. $A_1$, \ldots, $A_n$ entail $B$ iff for any valuation $v$ such that $v(A_1)=1$, \ldots, $v(A_n)$, it holds that $v(B)=1$.
\end{enumerate}
\end{definition}

We consider the following proof systems for \L{}$_3$. First, we borrow a~Frege system from~\cite{Avron1991}. Second, we augment it with the disjunction elimination rule that allows us to use assumptions. We take the idea and notational conventions for our second and third calculi from~\cite{BonetBuss1993}. Note that all these systems are sound and implicationally complete~\cite{Avron1991} and hence are Frege systems.

\begin{definition}[$\HLuk$ \cite{Avron1991}]\label{HLuk}
$\HLuk$ is a~Frege system for $\Luk_3$ which has the following axiom schemas.
\begin{enumerate}
\item $A\supset(B\supset A)$
\item $(A\supset B)\supset((B\supset C)\supset(A\supset C))$
\item $((A\supset B)\supset B)\supset((B\supset A)\supset A)$
\item $((((A\supset B)\supset A)\supset A)\supset(B\supset C))\supset(B\supset C)$
\item $(A\wedge B)\supset A$
\item $(A\wedge B)\supset B$
\item $(A\supset B)\supset((A\supset C)\supset(A\supset(B\wedge C)))$
\item $A\supset(A\vee B)$
\item $B\supset(A\vee B)$
\item $(A\supset C)\supset((B\supset C)\supset((A\vee B)\supset C))$
\item $(\neg B\supset\neg A)\supset(A\supset B)$
\end{enumerate}

The only rule of inference is modus ponens $\dfrac{A\quad A\supset B}{B}$.

We say that $B$ is derived in $\HLuk$ from a~set of formulas $\Gamma$ iff there is a~finite sequence of formulas $A_1,\ldots,A_n$ such that $A_n=A$ and each $A_i$ is either an axiom, a~member of $\Gamma$, or obtained from previous ones by modus ponens.

We say that there is an $\HLuk$-proof of $A$ iff $A$ can be derived from an empty set of formulas.
\end{definition}

The next two systems --- $\Luk_{3_\vee}$ and $\Luk_{3_{n\vee}}$ --- have \textit{sequents} of the form $\Gamma\vDash A$ instead of \textit{formulas} as steps in their derivations. They differ from each other in that left hand sides of sequents in $\Luk_{3_\vee}$-proofs are finite \textit{sets} of formulas, while left hand sides of sequents in $\Luk_{3_{n\vee}}$-proofs are finite \textit{sequences} of formulas.

Before tackling these systems we introduce two notions --- those of availability of a~sequent and of open assumptions.
\begin{definition}\label{availability}
We say that $\Gamma'\vDash A'$ is available to $\Gamma''\vDash A''$ iff and $\Gamma'$ is the initial subsequence of $\Gamma'$.
\end{definition}
\begin{definition}
We say that the assumption $C$ is open at the step $\Gamma\vDash A$ iff $C\in\Gamma$.
\end{definition}
\begin{definition}[$\Luk_{3_\vee}$]\label{orLuk}
$\Luk_{3_\vee}$ is the system akin to the general deduction Frege system from~\cite{BonetBuss1993} but with a~version of the disjunction elimination rule instead of the deduction rule. Each step in $\Luk_{3_\vee}$ proof is a~sequent of the form $\Gamma\Rightarrow A$, where $A$ is a~formula and $\Gamma$ is a~finite (or empty) set of formulas. If $\Gamma=\varnothing$, we write $\Rightarrow A$. 

$\Luk_{3_\vee}$ has two axioms:
\begin{itemize}
\item $\Rightarrow A$, where $A$ is an instance of an axiom schema for $\HLuk$;
\item $\{A\}\Rightarrow A$, where $A$ is arbitrary formula;
\end{itemize}
and two inference rules:
\begin{itemize}
\item $\dfrac{\Gamma_i\Rightarrow A\quad\Gamma_j\Rightarrow A\supset B}{\Gamma_i\cup\Gamma_j\Rightarrow B}$ --- modus ponens;
\item $\dfrac{\Gamma_i\Rightarrow A\vee B\quad\Gamma_j\cup\{A\}\Rightarrow C\quad\Gamma_k\cup\{B\}\Rightarrow C}{(\Gamma_i\cup\Gamma_j\cup\Gamma_k)\setminus\{A,B\}\Rightarrow C}$ --- $\vee_e$.
\end{itemize}

We say that $A$ is inferred from $\Gamma$ in $\Luk_{3_\vee}$ iff there is a~finite sequence of sequents $\Gamma_1\vDash A_1,\ldots,\Gamma_n\vDash A_n$ such that $\Gamma_n=\Gamma$ and $A_n=A$ and every sequent is either an axiom or obtained from previous ones by one of the inference rules.

We say that there is an $\Luk_{3_\vee}$-proof of $A$ iff $A$ can be derived from an empty set of formulas.
\end{definition}
\begin{definition}[$\Luk_{3_{n\vee}}$]\label{orLuknested}
$\Luk_{3_\vee}$ is the system akin to the nested deduction Frege system from~\cite{BonetBuss1993} but with a~version of the disjunction elimination rule instead of the deduction rule.

An $\Luk_{3_{n\vee}}$ proof is a~sequence of sequents $\Gamma_1\Rightarrow A_1$, \ldots, $\Gamma_n\Rightarrow A_n$, where $\Gamma_0$ is empty, each $\Gamma_i$ is a~finite sequence of formulas and for any $i$ one of the following holds.
\begin{enumerate}
\item $\Gamma_i=\Gamma_{i-1}$ and $A_i$ is an instance of an axiom schema for $\HLuk$.
\item $\Gamma_i=\Gamma_{k}*\langle A_i\rangle$ with $k$ being the last available sequent to $i$. This opens assumption $A_i$.
\item $A_i=C$ is inferred from $A_l=A\vee B$, $A_k=C$ and $A_j=C$, $\Gamma_k=\Gamma_i*\langle A\rangle$, $\Gamma_j=\Gamma_i*\langle B\rangle$ and $\Gamma_l\Rightarrow A_l$ is available to $\Gamma_i\Rightarrow A_i$. All sequents from the opening of $A$ as an assumption to $\Gamma_i*\langle A\rangle\Rightarrow C$ and all sequents from the opening of $B$ as an assumption to $\Gamma_i*\langle B\rangle\Rightarrow C$ become unavailable for further steps; assumption $A$ is closed at $k+1$'th step and assumption $B$ is closed at $j+1$'th step. This is the $\vee_{ne}$ rule.
\item $\Gamma_i=\Gamma_{i-1}$, $A_i$ is inferred from $A_j=A_k\supset A_i$ and $A_k$ by modus ponens with both $\Gamma_j\Rightarrow A_j$ and $\Gamma_k\Rightarrow A_k$ being available to $\Gamma_i\Rightarrow A_i$.
\end{enumerate}

We will further (just as in~\cite{BonetBuss1993}) represent nested derivations not as sequences of sequents but as columns of formulas with vertical bars denoting the opening and closing of assumptions and availability of steps of the derivation. So the implementation of the $\vee_{ne}$ rule looks as follows.
\begin{longtable}{l}
$A\vee B$\\
$\left[\begin{tabular}{l}
$A$ --- assumption\\
\vdots\\
$C$
\end{tabular}\right.$\\
$\left[\begin{tabular}{l}
$B$ --- assumption\\
\vdots\\
$C$
\end{tabular}\right.$\\
$C$ --- $\vee_{ne}$
\end{longtable}
\end{definition}

We again stress the fact that our proof systems are adapted for 3-valued \L{}u\-ka\-sie\-wicz logic from the ones proposed by Bonet and Buss in~\cite{BonetBuss1993} for classical logic. In particular, $\Luk_{3_\vee}$ is defined in the same fashion as the general deduction Frege system from~\cite{BonetBuss1993}. $\Luk_{3_{n\vee}}$ is constructed similarly to nested deduction Frege systems from~\cite{BonetBuss1993}. However, since the deduction theorem does not hold, we use general and nested versions of the disjunction elimination rule. We also adapted the way we open assumptions and the notion of an available sequent when we define $\Luk_{3_{n\vee}}$ which is $\Luk_3$ analogue to nested deduction Frege systems. Examples of calculi for classical and non-classical logics with such a~rule can be found in~\cite{Fitch1952,Kozhemiachenko2018LLP,BolotovKozhemiachenkoShangin2018}.

\begin{definition}\label{simulation}
We say that proof system $\mathcal{S}_2$ simulates proof system $\mathcal{S}_1$ if there is an algorithm that transforms any $\mathcal{S}_1$ proof of any formula $A$ into an $\mathcal{S}_2$ proof of $A$.

We say that simulation is polynomial (linear, quadratic, etc.) if $\mathcal{S}_1$ proof of $A$ in $n$ steps is transformed to $\mathcal{S}_2$ proof of $A$ in $O(g(n))$ steps with $g$ being a~polynomial (linear, quadratic, etc.) function.
\end{definition}
\section{Main results\label{mainresults}}
\subsection[Simulation of $\vee$-elimination in classical logic]{Simulation of the disjunction elimination rule in classical logic}
Recall the results of Bonet and Buss~\cite{BonetBuss1993} about upper bounds on speed-ups. A Frege system for classical logic can be taken, for example, from Kleene~\cite{Kleene1971}.
\begin{theorem}[Bonet and Buss~\cite{BonetBuss1993}]
There is a~constant $c$ such that if $B$ is inferred in a~Frege system from $A$ in $n$ steps, then $A\supset B$ has Frege proof in $c\cdot n$ steps.
\end{theorem}
\begin{theorem}[Bonet and Buss~\cite{BonetBuss1993}]
Suppose, $B$ is inferred from $A_1$, \ldots, $A_m$ in a~Frege system in $n$ steps. Then $A_1\supset(\ldots\supset(A_m\supset B))$ has Frege proof in $O(m+n)$ steps. 
\end{theorem}
\begin{theorem}[Bonet and Buss~\cite{BonetBuss1993}]
If $B$ has a~general deduction Frege proof in $n$ steps, then $B$ has Frege proof in $O(n^2)$ steps.
\end{theorem}
\begin{theorem}[Bonet and Buss~\cite{BonetBuss1993}]
Suppose $B$ has a~nested deduction Frege proof in $n$ steps, where assumptions are opened $m$ times, then $B$ has Frege proof in $O(n+m\log^{(*i)}m)$ steps.
\end{theorem}
\begin{theorem}[Bonet and Buss~\cite{BonetBuss1993}]
If $B$ has a~nested deduction Frege proof in $n$ steps. Then $B$ has Frege proof in $O(n\alpha(n))$ steps.
\end{theorem}

The deduction theorem\footnote{In fact, even the version of the deduction theorem $A\vDash B\text{ iff }\vDash\neg A\vee B$ as given in, for instance,~\cite{Shoenfield1967}, does not hold in $\Luk_3$.}, however, does not hold in $\Luk_3$, so it is instructive to consider proof systems obtained from the Frege system for classical logic by augmenting it with the disjunction elimination rule rather than the deduction rule. We will call these systems \textit{nested disjunction Frege system} ($no\mathscr{F}$) and \textit{general disjunction Frege system} ($o\mathscr{F}$), respectively.
\begin{definition}[$o\mathscr{F}$]
$o\mathscr{F}$ is a~system akin to the general deduction Frege system from~\cite{BonetBuss1993} but with a~version of the disjunction elimination rule instead of the deduction rule. Each step in $o\mathscr{F}$ proof is a~sequent of the form $\Gamma\Rightarrow A$, where $A$ is a~formula and $\Gamma$ is a~finite (or empty) set of formulas. If $\Gamma=\varnothing$, we write $\Rightarrow A$. 

$o\mathscr{F}$ has two axioms:
\begin{itemize}
	\item $\Rightarrow A$, where $A$ is an instance of an axiom schema for a~classical Frege system;
	\item $\{A\}\Rightarrow A$, where $A$ is arbitrary formula;
\end{itemize}
and two inference rules:
\begin{itemize}
	\item $\dfrac{\Gamma_i\Rightarrow A\quad\Gamma_j\Rightarrow A\supset B}{\Gamma_i\cup\Gamma_j\Rightarrow B}$ --- modus ponens;
	\item $\dfrac{\Gamma_i\Rightarrow A\vee B\quad\Gamma_j\cup\{A\}\Rightarrow C\quad\Gamma_k\cup\{B\}\Rightarrow C}{(\Gamma_i\cup\Gamma_j\cup\Gamma_k)\setminus\{A,B\}\Rightarrow C}$ --- $\vee_e$.
\end{itemize}

We say that $A$ is inferred from $\Gamma$ in $o\mathscr{F}$ iff there is a~finite sequence of sequents $\Gamma_1\vDash A_1,\ldots,\Gamma_n\vDash A_n$ such that $\Gamma_n=\Gamma$ and $A_n=A$ and every sequent is either an axiom or obtained from previous ones by one of the inference rules.

We say that there is an $o\mathscr{F}$-proof of $A$ iff $A$ can be derived from an empty set of formulas.
\end{definition}
\begin{definition}[$no\mathscr{F}$]
$no\mathscr{F}$ is the system akin to the nested deduction Frege system from~\cite{BonetBuss1993} but with a~version of the disjunction elimination rule instead of the deduction rule.

An $no\mathscr{F}$ proof is a~sequence of sequents $\Gamma_1\Rightarrow A_1$, \ldots, $\Gamma_n\Rightarrow A_n$, where $\Gamma_0$ is empty, each $\Gamma_i$ is a~finite sequence of formulas and for any $i$ one of the following holds.
\begin{enumerate}
\item $\Gamma_i=\Gamma_{i-1}$ and $A_i$ is an instance of an axiom schema for a~classical Frege system.
\item $\Gamma_i=\Gamma_{k}*\langle A_i\rangle$ with $k$ being the last available (cf.~definition~\ref{availability}) sequent to $i$. This opens assumption $A_i$.
\item $A_i=C$ is inferred from $A_l=A\vee B$, $A_k=C$ and $A_j=C$, $\Gamma_k=\Gamma_i*\langle A\rangle$, $\Gamma_j=\Gamma_i*\langle B\rangle$ and $\Gamma_l\Rightarrow A_l$ is available to $\Gamma_i\Rightarrow A_i$. All sequents from the opening of $A$ as an assumption to $\Gamma_i*\langle A\rangle\Rightarrow C$ and all sequents from the opening of $B$ as an assumption to $\Gamma_i*\langle B\rangle\Rightarrow C$ become unavailable for further steps; assumption $A$ is closed at $k+1$'th step and assumption $B$ is closed at $j+1$'th step. This is the $\vee_{ne}$ rule.
\item $\Gamma_i=\Gamma_{i-1}$, $A_i$ is inferred from $A_j=A_k\supset A_i$ and $A_k$ by modus ponens with both $\Gamma_j\Rightarrow A_j$ and $\Gamma_k\Rightarrow A_k$ being available to $\Gamma_i\Rightarrow A_i$.
\end{enumerate}

We will further (just as in~\cite{BonetBuss1993}) represent nested derivations not as sequences of sequents but as columns of formulas with vertical bars denoting opening and closing of assumptions and availability of steps of the derivation. So the implementation of the $\vee_{ne}$ rule looks as follows.
\begin{longtable}{l}
$A\vee B$\\
$\left[\begin{tabular}{l}
$A$ --- assumption\\
\vdots\\
$C$
\end{tabular}\right.$\\
$\left[\begin{tabular}{l}
$B$ --- assumption\\
\vdots\\
$C$
\end{tabular}\right.$\\
$C$ --- $\vee_{ne}$
\end{longtable}
\end{definition}

We now provide upper bounds on speed-ups of $o\mathscr{F}$ and $no\mathscr{F}$ over Frege systems for classical logic. To this end we will prove that $o\mathscr{F}$ and $d\mathscr{F}$ as well as $no\mathscr{F}$ and $nd\mathscr{F}$ linearly simulate one another. We could have provided a~direct simulation of $o\mathscr{F}$ and $no\mathscr{F}$ by Frege systems but our approach relates the deduction theorem with the disjunction elimination rule in classical logic.
\begin{theorem}\label{ndf=nof}
$nd\mathscr{F}$ and $no\mathscr{F}$ linearly simulate one another.
\end{theorem}
\begin{proof}
Observe that the only difference between $nd\mathscr{F}$ and $no\mathscr{F}$ is the fact that the first one has nested deduction rule and the second one --- nested disjunction elimination rule. That is why to prove the theorem it suffices to show that we can simulate the deduction rule via the disjunction elimination rule and vice versa in a~constant number of steps.

Indeed, take the implementation of the disjunction elimination rule.

\begin{longtable}{l}
$A\vee B$\\
$\left[\begin{tabular}{l}
$A$ --- assumption\\
\vdots\\
$C$ --- from $A$
\end{tabular}\right.$\\
$\left[\begin{tabular}{l}
$B$ --- assumption\\
\vdots\\
$C$ --- from $A$
\end{tabular}\right.$\\
$C$ --- disjunction elimination rule
\end{longtable}

Our simulation is straightforward since $(A\supset C)\supset((B\supset C)\supset((A\vee B)\supset C))$ is an axiom. The details are left to the reader.
\begin{comment}
\begin{longtable}{l}
$A\vee B$\\
$\left[\begin{tabular}{l}
$A$ --- assumption\\
\vdots\\
$C$ --- from $A$
\end{tabular}\right.$\\
$A\supset C$ --- deduction rule\\
$\left[\begin{tabular}{l}
$B$ --- assumption\\
\vdots\\
$C$ --- from $A$
\end{tabular}\right.$\\
$B\supset C$ --- deduction rule\\
$(A\supset C)\supset((B\supset C)\supset((A\vee B)\supset C))$ --- axiom\\
$(B\supset C)\supset((A\vee B)\supset C)$ --- modus ponens\\
$(A\vee B)\supset C$ --- modus ponens\\
$C$ --- modus ponens
\end{longtable}
\end{comment}

The reverse simulation is also straightforward. Assume we have the following implementation of nested deduction rule.
\begin{longtable}{l}
$\left[\begin{tabular}{l}
$A$ --- assumption\\
\vdots\\
$B$ --- from $A$
\end{tabular}\right.$\\
$A\supset B$ --- deduction rule\\
\end{longtable}

We proceed as follows.
\begin{longtable}{l}
\vdots\\
$A\vee\neg A$ --- in a~constant number of steps\\
$\left[\begin{tabular}{l}
$A$ --- assumption\\
\vdots\\
$B$ --- from $A$\\
\vdots\\
$A\supset B$ --- in a~constant number of steps from $B$
\end{tabular}\right.$\\
$\left[\begin{tabular}{l}
$\neg A$ --- assumption\\
\vdots\\
$A\supset B$ --- in a~constant number of steps from $\neg A$
\end{tabular}\right.$\\
$A\supset B$ --- disjunction elimination rule\\
\end{longtable}
\end{proof}
\begin{theorem}
$d\mathscr{F}$ and $o\mathscr{F}$ linearly simulate one another.
\end{theorem}
\begin{proof}
Analogously to Theorem~\ref{ndf=nof}.
\end{proof}

An immediate corollary is as follows.
\begin{corollary}
\begin{enumerate}
\item[]
\item If $B$ has a~general disjunction Frege proof in $n$ steps, then $B$ has a~Frege proof in $O(n^2)$ steps.
\item Suppose $B$ has a~nested disjunction Frege proof in $n$ steps, where assumptions are opened $m$ times, then $B$ has a~Frege proof in $O(n+m\log^{(*i)}m)$ steps.
\item If $B$ has a~nested disjunction Frege proof in $n$ steps. Then $B$ has a~Frege proof in $O(n\alpha(n))$ steps.
\end{enumerate}
\end{corollary}
\subsection[Simulation $\vee$-elimination in $\Luk_3$]{Simulation of the disjunction elimination rule in $\Luk_3$}
First, we prove a~result akin to the deduction theorem for $\Luk_3$. Observe that the following statement is true:
\[\begin{array}{ccccc}
A\vDash B&\text{iff}&\vDash A\supset(A\supset B)&\text{iff}&\vDash\neg(A\supset\neg A)\supset B
\end{array}\]
\begin{lemma}\label{Lukdeduction}
Assume, $B$ is derived from $A$ in $n$ steps. Then there is a~constant $c_1$ such that $A\supset(A\supset B)$ has an {$\HLuk$} proof in $c_1\cdot n$ steps and there is a~constant $c_2$ such that there is an {$\HLuk$} proof of $\neg(A\supset\neg A)\supset B$ in $c_2\cdot n$ steps.
\end{lemma}
\begin{proof}
Observe, first, that there is an $\HLuk$ derivation of $\neg(A\supset A)\supset B$ from $A\supset(A\supset B)$ in a~constant number of steps. It thus suffices to show that the lemma holds for $A\supset(A\supset B)$.

We obtain the constant bound as follows. First, we replace all formulas $C$ with $A\supset(A\supset C)$ in a~derivation of $B$ from $A$. Now we need to fill in the gaps in a~constant number of steps for each gap. $A$ itself becomes $A\supset(A\supset A)$ which has an $\HLuk$ proof in a~constant number of steps. If $C$ is an axiom, then $A\supset(A\supset C)$ also has a~proof in a~constant number of steps. Finally, if $C$ is inferred from $B$ and $B\supset C$ by modus ponens, we obtain $A\supset(A\supset(B\supset C))$ and $A\supset(A\supset B)$. From them, we can infer $A\supset(A\supset C)$ in a~constant number of steps.
\end{proof}

Our next theorem shows that both
\[\dfrac{A_1,\ldots,A_m\vdash B}{\vdash A_1\supset (A_1\supset\ldots\supset(A_m\supset(A_m\supset B)))}\]
and
\[\dfrac{A_1,\ldots,A_m\vdash B}{\vdash(\neg(A_1\supset\neg A_1)\wedge\ldots\wedge\neg(A_m\supset\neg A_m))\supset B}\]
rules provide at most linear speed-up.

\begin{theorem}\label{HLuksimplededuction}
Assume there is an {$\HLuk$} derivation of $B$ from $A_1$, \ldots, $A_m$ in $n$ steps. Then there are {$\HLuk$} proofs in $O(m+n)$ steps of $A_1\supset (A_1\supset\ldots\supset(A_m\supset(A_m\supset B)))$ and $(\neg(A_1\supset\neg A_1)\wedge\ldots\wedge\neg(A_m\supset\neg A_m))\supset B$, where conjunction is associated arbitrarily.
\end{theorem}
\begin{proof}
Observe that it can be proved by induction on $m$ that there is a~derivation of $A_1\supset (A_1\supset\ldots\supset(A_m\supset(A_m\supset B)))$ from $(\neg(A_1\supset\neg A_1)\wedge\ldots\wedge\neg(A_m\supset\neg A_m))\supset B$ in $O(m)$ steps.

We construct a~derivation of $B$ from a~single assumption $\neg(A_1\supset\neg A_1)\wedge\ldots\wedge\neg(A_m\supset\neg A_m)$. Since $(C_1\wedge C_2)\supset C_i$ is an axiom for $i=1,2$, we need a~constant number of steps (say, $c_1$) to derive $C_i$ from $C_1\wedge C_2$. Hence we will need $2c_1(m-1)$ steps to infer $\neg(A_i\supset\neg A_i)$ for $i=1,\ldots,m$.

Now, there is a~derivation of $A$ from $\neg(A\supset\neg A)$ for any $A$ in a~constant number (say, $c_2$) of steps. So, we need $2c_1(m-1)+c_2m$ steps to infer $A_1$, \ldots, $A_n$. Now we need $n$ steps to infer $B$. Hence, by Lemma~\ref{Lukdeduction} we need $O(m+n)$ steps to infer $$(\neg(A_1\supset\neg A_1)\wedge\ldots\wedge\neg(A_m\supset\neg A_m))\supset((\neg(A_1\supset\neg A_1)\wedge\ldots\wedge\neg(A_m\supset\neg A_m))\supset B)$$ From here there is an $\HLuk$ inference of $$(\neg(A_1\supset\neg A_1)\wedge\ldots\wedge\neg(A_m\supset\neg A_m))\supset B$$ in $O(m)$ steps.
\end{proof}
\begin{lemma}\label{HLukextraction}
Suppose that $\mathfrak{B}$ is an arbitrarily associated conjunction $A_1\wedge\ldots\wedge A_m$ and $\mathfrak{C}$ is some arbitrarily associated conjunction $A_{i_1}\wedge\ldots\wedge A_{i_n}$, where $i_{1}$, \ldots, $i_{n}$ is any sequence from $\{1,\ldots,m\}$.  Then there is an {$\HLuk$} proof of $\mathfrak{B}\supset\mathfrak{C}$ in $O(n\log m)$ steps.
\end{lemma}
\begin{proof}
It can be easily showed by induction on $m$ that $(A_1\wedge\ldots\wedge A_m)\supset A_k$, where $k\leqslant m$, has an $\HLuk$ proof in $O(\log m)$ steps. Indeed, we simply employ binary search for $A_k$ (recall that $(A_1\wedge A_2)\supset A_i$ is an axiom) and obtain $O(\log m)$ following formulas.
\begin{eqnarray}
(A_1\wedge\ldots\wedge A_m)\supset(A_i\wedge\ldots\wedge A_k\wedge\ldots A_l)\nonumber\\
\nonumber\vdots\\
\nonumber(A_j\wedge A_k)\supset A_k
\end{eqnarray}

These formulas are axioms\footnote{Here $A_i\wedge\ldots\wedge A_k\wedge\ldots A_l$ ($i,k,l\leqslant m$) is the conjunct of $A_1\wedge\ldots\wedge A_m$ containing $A_k$.}, and thus are proved in a~constant number of steps. Recall, that for any $A$, $B$ and $C$ there is an $\HLuk$ derivation in a~constant number of steps of $A\supset C$ from $A\supset B$ and $B\supset C$. We apply this fact $O(\log m)$ times.

Now, since for any $D$, $E$ and $F$ there is an $\HLuk$ derivation of $D\supset(E\wedge F)$ from $D\supset E$ and $D\supset F$ in a~constant number of steps, and since we need to infer $(A_1\wedge\ldots\wedge A_m)\supset A_k$ for all $k\in\{i_1,\ldots,i_n\}$, we need $O(n\log m)$ steps to prove $\mathfrak{B}\supset\mathfrak{C}$.
\end{proof}

Observe that we would have needed only $O(m+n)$ steps to prove $\mathfrak{B}\supset\mathfrak{C}$ had we been allowed to use the deduction Theorem~\cite{BonetBuss1993}.
\begin{theorem}\label{HLuktoorLuk}
Assume there is a~$\Luk{3_\vee}$ proof of $D$ in $n$ steps. Then there is a~{$\HLuk$} proof of $D$ in $O(n^2\log n)$ steps.
\end{theorem}
\begin{proof}
In the proof of this theorem we will denote by $\bigotimes\limits^{m}_{i=1}A_i$ a~conjunction of $m$ formulas $\neg(A_i\supset\neg A_i)$ ordered and associated arbitrarily.

We quickly note that since $\HLuk$ is sound and complete and $\vee_e$ is a~valid rule, $\Luk{3_\vee}$ is also a~sound and complete calculus. Hence if we prove $\{C_1,\ldots,C_n\}\Rightarrow D$, it is the case that $C_1,\ldots,C_n\vDash D$. Moreover, it is also the case that $\vDash\bigotimes\limits^{n}_{i=1}C_i\supset D$.

We will now prove a~more general result, namely, that if there are $n$ steps in the $\Luk{3_\vee}$ proof of $\{A_1,\ldots,A_m\}\Rightarrow B$, then there are $O(n^2\log n)$ steps in the $\HLuk$ proof of $\bigotimes\limits^{m}_{i=1}A_i\supset B$.

To prove the result we will do the following. First, we replace each sequent $\{C_1,\ldots,C_n\}\Rightarrow D$ in $\Luk{3_\vee}$ proof and replace it with $\bigotimes\limits^{n}_{i=1}C_i\supset D$. We will then fill in the gaps. It suffices to show that we need $O(n\log n)$ steps to fill in every gap .

The proof now splits into the cases depending on how $\bigotimes\limits^{m}_{i=1}A_i\supset B$ is inferred.

First, an axiom $\Rightarrow A$ in $\Luk{3_\vee}$ proof becomes an axiom of $\HLuk$. Second, an axiom $\{A\}\Rightarrow A$ becomes $\neg(A\supset\neg A)\supset A$ which has an $\HLuk$ proof in a~constant number of steps.

Third, sequents in a~modus ponens inference 
$$\dfrac{\Gamma_1\Rightarrow A\quad\Gamma_2\Rightarrow A\supset B}{\Gamma_1\cup\Gamma_2\Rightarrow B}$$
become $\bigotimes\Gamma_1\supset A$, $\bigotimes\Gamma_2\supset(A\supset B)$ and $\bigotimes(\Gamma_1\cup\Gamma_2)\supset B$.

We now need to show that there is a~derivation in $O(n\log n)$ steps of $\bigotimes\Gamma_1\supset A$ from both $\bigotimes\Gamma_2\supset(A\supset B)$ and $\bigotimes(\Gamma_1\cup\Gamma_2)\supset B$. Observe that there are at most $n$ formulas in $\Gamma_1\cup\Gamma_2$. By Lemma~\ref{HLukextraction}, there are $\HLuk$ proofs of $\bigotimes(\Gamma_1\cup\Gamma_2)\supset\bigotimes\Gamma_1$ and $\bigotimes(\Gamma_1\cup\Gamma_2)\supset\bigotimes\Gamma_2$ in $O(n\log n)$ steps each. From here we obtain $\bigotimes(\Gamma_1\cup\Gamma_2)\supset A$ and $\bigotimes(\Gamma_1\cup\Gamma_2)\supset(A\supset B)$ in a~constant number of steps which allows us to finally infer $\bigotimes(\Gamma_1\cup\Gamma_2)\supset B$ in a~constant number of steps.

Finally, consider the case of $\vee_e$. The sequents in $$\dfrac{\Gamma_1\Rightarrow A\vee B\quad\Gamma_2\cup\{A\}\Rightarrow C\quad\Gamma_3\cup\{B\}\Rightarrow C}{(\Gamma_1\cup\Gamma_2\cup\Gamma_3)\setminus\{A,B\}\Rightarrow C}$$ become $\bigotimes\Gamma_1\supset(A\vee B)$, $\bigotimes(\Gamma_2\cup\{A\})\supset C$, $\bigotimes(\Gamma_3\cup\{B\})\supset C$ and $\bigotimes((\Gamma_1\cup\Gamma_2\cup\Gamma_3)\setminus\{A,B\})\supset C$. Again, by Lemma~\ref{HLukextraction} and the fact that the number of formulas in $\Gamma_1\cup\Gamma_2\cup\Gamma_3$ is bounded by $n$, there are $\HLuk$ proofs of
\begin{eqnarray}
((\neg(A\supset\neg A)\wedge\neg(B\supset\neg B))\wedge\bigotimes((\Gamma_1\cup\Gamma_2\cup\Gamma_3)\setminus\{A,B\}))\supset\bigotimes\Gamma_1\nonumber\\
(\neg(A\supset\neg A)\wedge\bigotimes((\Gamma_1\cup\Gamma_2\cup\Gamma_3)\setminus\{A,B\}))\supset\bigotimes(\Gamma_2\cup\{A\})\nonumber\\
(\neg(B\supset\neg B)\wedge\bigotimes((\Gamma_1\cup\Gamma_2\cup\Gamma_3)\setminus\{A,B\}))\supset\bigotimes(\Gamma_2\cup\{B\})\nonumber
\end{eqnarray}
in $O(n\log n)$ steps each.

From here we in constant number of steps obtain
\begin{eqnarray}
((\neg(A\supset\neg A)\wedge\neg(B\supset\neg B))\wedge\bigotimes((\Gamma_1\cup\Gamma_2\cup\Gamma_3)\setminus\{A,B\}))\supset(A\vee B)\nonumber\\
(\neg(A\supset\neg A)\wedge\bigotimes((\Gamma_1\cup\Gamma_2\cup\Gamma_3)\setminus\{A,B\}))\supset C\nonumber\\
(\neg(B\supset\neg B)\wedge\bigotimes((\Gamma_1\cup\Gamma_2\cup\Gamma_3)\setminus\{A,B\}))\supset C\nonumber
\end{eqnarray}

From here we finally infer $\bigotimes((\Gamma_1\cup\Gamma_2\cup\Gamma_3)\setminus\{A,B\})\supset C$ in a~constant number of steps.
\end{proof}

Recall that Frege systems for classical logic simulate general deduction Frege systems \textit{quadratically}~\cite{BonetBuss1993}, so we can see a~slight discrepancy between $\Luk_3$ and classical case. However, our next theorem states that $\HLuk$ simulates $\Luk{3_{n\vee}}$ nearly linearly in the same way as Frege systems for classical logics simulate nested deduction Frege systems~\cite{BonetBuss1993}.
\begin{theorem}\label{HLuktoorLuknested}
\begin{enumerate}
\item[]
\item Assume there is an $\Luk{3_{n\vee}}$ proof of $A$ in $n$ steps. Then there is an {$\HLuk$} proof of $A$ in $O(n\cdot\alpha(n))$ steps with $\alpha$ being inverse Ackermann function as defined in~\cite{BonetBuss1993} and~\cite{BonetBuss1995}.
\item Assume there is an $\Luk{3_{n\vee}}$ proof of $A$ in $n$ steps with assumptions been opened $m$ times. Then there is an {$\HLuk$} proof of $A$ in $O(n\cdot m\log^{(*i)}m)$ steps.
\end{enumerate}
\end{theorem}
\begin{proof}
Just as in~\cite{BonetBuss1993} we will reduce the theorem to the transitive closure of trees. We now need to translate an $\Luk{3_\vee}$ proof of $A$ in $n$ steps into an $\HLuk$ proof of $A$ in $O(n\cdot\alpha(n))$ steps to prove (1) and into an $\HLuk$ proof of $A$ in $O(n\cdot m\log^{(*i)}m)$ steps to prove (2).

Our simulation proceeds as follows. From each sequent $\Gamma\Rightarrow A$ ($\Gamma=\langle A_1,\ldots,A_k\rangle$) in the $\Luk{3_\vee}$ proof we form a~logically equivalent formula $\bigotimes\Gamma\supset A$, where $\bigotimes\Gamma$ is $\neg(A_1\supset\neg A_1)\wedge\ldots\wedge\neg(A_k\supset A_k)$ associated to the right. If $\Gamma$ is empty, $\bigotimes\Gamma$ is some fixed valid formula.

This translation transforms
\[\Gamma_1\Rightarrow A_1,\ldots,\Gamma_n\Rightarrow A_n\]
to
\[\bigotimes\Gamma_1\supset A_1,\ldots,\bigotimes\Gamma_n\supset A_n\]

It remains now to fill in the gaps faster than in linear time.

If $\Gamma\Rightarrow A$ is an axiom, $\bigotimes\Gamma\supset A$ can be proved in a~constant number of steps.

If $\Gamma*\langle A\rangle\Rightarrow A$ is inferred as an assumption, then there is an $\HLuk$ proof of $(\bigotimes\Gamma\wedge\neg(A\supset\neg A))\supset A$ in a~constant number of steps.

Now, consider modus ponens. There are sequents $\Gamma_1\Rightarrow A$ and $\Gamma_2\Rightarrow A\supset B$ from which the sequent $\Gamma\Rightarrow B$ is inferred. But since we work in $\Luk{3_{n\vee}}$, $\Gamma_1$ and $\Gamma_2$ are initial subsequences of $\Gamma$. Observe that there is an $\HLuk$ derivation of $\bigotimes\Gamma\supset B$ from $\bigotimes\Gamma_1\supset A$, $\bigotimes\Gamma_2\supset(A\supset B)$ and $\bigotimes\Gamma\supset\bigotimes\Gamma_i$ for $i=1,2$ in a~constant number of steps.

Finally we tackle the $\vee_{ne}$ case. This means that we have sequents $\Gamma\Rightarrow A\vee B$, $\Gamma*\langle A\rangle\Rightarrow C$ and $\Gamma*\langle B\rangle\Rightarrow C$ from which we infer $\Gamma\Rightarrow C$. We translate these sequents into $\bigotimes\Gamma\supset(A\vee B)$, $(\bigotimes\Gamma\wedge\neg(A\supset\neg A))\supset C$, $(\bigotimes\Gamma\wedge\neg(B\supset\neg B))\supset C$ and $\bigotimes\Gamma\supset C$. It can be easily seen that there is an $\HLuk$ derivation of the fourth formula from the first three ones in a~constant number of steps.

It now remains to show that there are short $\HLuk$ proofs of $\bigotimes\Gamma\supset\bigotimes\Gamma_i$. It is straightforward that there are at most $2n$ such formulas. Note that $\bigotimes\Gamma_i$ is always an initial subsequence of $\bigotimes\Gamma$ and $\bigotimes\Gamma\supset\bigotimes\Gamma_i$ has the following form
\[((\ldots(A_1\wedge A_2)\wedge\ldots)\wedge A_k)\supset((\ldots(A_1\wedge A_2)\wedge\ldots)\wedge A_l).\]
To prove one such formula we need $O(k-l)$ steps. However, this will lead to quadratic simulation. To obtain near-linear simulation we need to proceed in the same way as in Theorems~5 and~6 from~\cite{BonetBuss1993}.

Recall that if there are $m$ assumptions, there are at most $m+1$ different $\bigotimes\Gamma$'s. Observe that $\bigotimes\Gamma$'s form a~directed tree with edges from $\bigotimes\Gamma$ to $\bigotimes\Gamma'$ precisely when $\Gamma$ exceeds $\Gamma'$ by one element. Observe that $\bigotimes\Gamma\supset\bigotimes\Gamma'$ is an axiom of $\HLuk$. Hence, it can be proved in one step. This gives us all edges in our directed tree.

Now, since implication in $\Luk_3$ is transitive (as well as in classical logic), since we have no more than $2n$ $\bigotimes\Gamma\supset\bigotimes\Gamma_i$ formulas to prove and since all $\bigotimes\Gamma\supset\bigotimes\Gamma'$ tautologies form a~directed tree, we can by the serial transitive closure~\cite{BonetBuss1995} infer all our $\bigotimes\Gamma\supset\bigotimes\Gamma_i$ formulas in $O(n\cdot\alpha(n))$ or $O(n\cdot m\log^{(*i)}m)$ steps.

We have thus reduced our theorem to the case of Theorems~5 and~6 from~\cite{BonetBuss1993}.
\end{proof}
\section{Simulation of natural deduction and hypersequent calculus\label{NDandGLuk}}
We start by defining natural deduction for $\Luk_3$ which we will further denote as $\mathbf{ND}_{\Luk_3}$ and the hypersequent calculus $\GLuk$.
\begin{definition}[$\mathbf{ND}_{\Luk_3}$~\cite{Tamminga2014}]\label{NDL3}
Tamminga defines derivations in $\mathbf{ND}_{\Luk_3}$ as follows. First, proof trees with either a~single assumption of $A$ or a~single axiom $$(A\vee\neg A)\vee((B\vee\neg B)\vee(A\supset B))$$ are derivations. Second, if $\mathcal{D}$, $\mathcal{D}_1$, $\mathcal{D}_2$, $\mathcal{D}_3$ are derivations\footnote{The notational conventions are borrowed from~\cite{TroelstraSchwichtenberg2000}.}, they can be extended by the following rules\footnote{Below double lines indicate that rules work in both directions.}.
\begin{center}
\AxiomC{$\mathcal{D}_1$}
\noLine
\UnaryInfC{$A$}
\AxiomC{$\mathcal{D}_2$}
\noLine
\UnaryInfC{$\neg A$}
\LeftLabel{EFQ}
\BinaryInfC{$B$}
\DisplayProof
\AxiomC{$\mathcal{D}_1$}
\noLine
\UnaryInfC{$A$}
\AxiomC{$\mathcal{D}_2$}
\noLine
\UnaryInfC{$B$}
\LeftLabel{$\wedge I$}
\BinaryInfC{$A\wedge B$}
\DisplayProof
\AxiomC{$\mathcal{D}$}
\noLine
\UnaryInfC{$A\wedge B$}
\LeftLabel{$\wedge E_1$}
\UnaryInfC{$A$}
\DisplayProof
\AxiomC{$\mathcal{D}$}
\noLine
\UnaryInfC{$A\wedge B$}
\LeftLabel{$\wedge E_2$}
\UnaryInfC{$B$}
\DisplayProof
\end{center}

\vspace{,5em}

\begin{center}
\AxiomC{$\mathcal{D}$}
\noLine
\UnaryInfC{$A$}
\LeftLabel{$\vee I_1$}
\UnaryInfC{$A\vee B$}
\DisplayProof
\AxiomC{$\mathcal{D}$}
\noLine
\UnaryInfC{$B$}
\LeftLabel{$\vee I_2$}
\UnaryInfC{$A\vee B$}
\DisplayProof
\AxiomC{$\mathcal{D}_1$}
\noLine
\UnaryInfC{$A\vee B$}
\AxiomC{$[A]^u$}
\noLine
\UnaryInfC{$\mathcal{D}_2$}
\noLine
\UnaryInfC{$C$}
\AxiomC{$[B]^v$}
\noLine
\UnaryInfC{$\mathcal{D}_3$}
\noLine
\UnaryInfC{$C$}
\LeftLabel{$\vee E^{u,v}$}
\TrinaryInfC{$C$}
\DisplayProof
\end{center}

\vspace{,5em}

\begin{center}
\AxiomC{$\mathcal{D}$}
\noLine
\UnaryInfC{$A$}
\LeftLabel{DN}
\doubleLine
\UnaryInfC{$\neg\neg A$}
\DisplayProof
\AxiomC{$\mathcal{D}$}
\noLine
\UnaryInfC{$\neg(A\vee B)$}
\LeftLabel{DeM$_\vee$}
\doubleLine
\UnaryInfC{$\neg A\wedge\neg B$}
\DisplayProof
\AxiomC{$\mathcal{D}$}
\noLine
\UnaryInfC{$\neg(A\wedge B)$}
\LeftLabel{DeM$_\wedge$}
\doubleLine
\UnaryInfC{$\neg A\vee\neg B$}
\DisplayProof
\end{center}

\vspace{,5em}

\begin{center}
\AxiomC{$\mathcal{D}$}
\noLine
\UnaryInfC{$\neg A\wedge\neg B$}
\LeftLabel{$\supset_1$}
\UnaryInfC{$A\supset B$}
\DisplayProof
\AxiomC{$\mathcal{D}$}
\noLine
\UnaryInfC{$\neg A$}
\LeftLabel{$\supset_2$}
\UnaryInfC{$(B\vee\neg B)\vee(A\supset B)$}
\DisplayProof
\AxiomC{$\mathcal{D}$}
\noLine
\UnaryInfC{$\neg A\wedge B$}
\LeftLabel{$\supset_3$}
\UnaryInfC{$A\supset B$}
\DisplayProof
\end{center}

\vspace{.5em}

\begin{center}
\AxiomC{$\mathcal{D}_1$}
\noLine
\UnaryInfC{$\neg B$}
\noLine
\AxiomC{$\mathcal{D}_2$}
\noLine
\UnaryInfC{($A\supset B)\vee\neg(A\supset B)$}
\LeftLabel{$\supset_4$}
\BinaryInfC{$A\vee\neg A$}
\DisplayProof
\AxiomC{$\mathcal{D}$}
\noLine
\UnaryInfC{$B$}
\LeftLabel{$\supset_5$}
\UnaryInfC{$(A\vee\neg A)\vee(A\supset B)$}
\DisplayProof
\AxiomC{$\mathcal{D}$}
\noLine
\UnaryInfC{$A\wedge\neg B$}
\LeftLabel{$\supset_6$}
\UnaryInfC{$\neg(A\supset B)$}
\DisplayProof
\end{center}

\vspace{.5em}

\begin{center}
\AxiomC{$\mathcal{D}_1$}
\noLine
\UnaryInfC{$A$}
\noLine
\AxiomC{$\mathcal{D}_2$}
\noLine
\UnaryInfC{($A\supset B)\vee\neg(A\supset B)$}
\LeftLabel{$\supset_7$}
\BinaryInfC{$B\vee\neg B$}
\DisplayProof
\AxiomC{$\mathcal{D}$}
\noLine
\UnaryInfC{$A\wedge B$}
\LeftLabel{$\supset_8$}
\UnaryInfC{$A\supset B$}
\DisplayProof
\end{center}

We will say that there is a~proof of $D$ if there is a~proof tree whose root is $D$, where all assumptions are closed. We will say that $D$ is inferred from $D_1,\ldots,D_n$ if there is a~proof tree whose root is $D$ and $D_1,\ldots,D_n$ are the open assumptions.
\end{definition}

To the best of our knowledge, there is no sequent calculus for $\Luk_3$. There is, nevertheless, an elegant formulation of hypersequent calculus $\GLuk$ provided by Avron in~\cite{Avron1991}.
\begin{definition}[$\GLuk$~\cite{Avron1991}]\label{GLuk} First, we define a~hypersequent as the construction \[\Gamma_1\Rightarrow\Delta_1\mid\ldots\mid\Gamma_n\Rightarrow\Delta_n\] where $\Gamma_i$ and $\Delta_i$ are finite (and possibly empty) sequences of formulas and $\Gamma_i\Rightarrow\Delta_i$ are called components. We will further use letters $\mathcal{G}$ and $\mathcal{H}$ (with indices when needed) to denote arbitrary hypersequents.

The only axiom of $\GLuk$ is $A\Rightarrow A$, where $A$ is an arbitrary formula.

Structural rules can be external (denoted with \textbf{E}), i.e., they work with components, and internal (denoted with \textbf{I}), i.e., they work with cedents of components.
\[\mathbf{EW}:\dfrac{\mathcal{G}}{\mathcal{G}\mid\mathcal{H}};\quad\mathbf{EC}:\dfrac{\mathcal{G}\mid\Gamma\Rightarrow\Delta\mid\Gamma\Rightarrow\Delta}{\mathcal{G}\mid\Gamma\Rightarrow\Delta};\quad\mathbf{EP}:\dfrac{\mathcal{G}\mid\Gamma_1\Rightarrow\Delta_1\mid\Gamma_2\Rightarrow\Delta_2\mid\mathcal{H}}{\mathcal{G}\mid\Gamma_2\Rightarrow\Delta_2\mid\Gamma_1\Rightarrow\Delta_1\mid\mathcal{H}}\]

\[\mathbf{IW}:\dfrac{\mathcal{G}\mid\Gamma\Rightarrow\Delta}{\mathcal{G}\mid\Gamma\Rightarrow\Delta,A},\quad\dfrac{\mathcal{G}\mid\Gamma\Rightarrow\Delta}{\mathcal{G}\mid A,\Gamma\Rightarrow\Delta};\]

\[\mathbf{IP}:\dfrac{\mathcal{G}\mid\Gamma\Rightarrow\Delta_1,A,B,\Delta_2}{\mathcal{G}\mid\Gamma\Rightarrow\Delta_1,B,A,\Delta_2},\quad\dfrac{\mathcal{G}\mid\Gamma_1,A,B,\Gamma_2\Rightarrow\Delta}{\mathcal{G}\mid\Gamma_1,B,A,\Gamma_2\Rightarrow\Delta}\]

\[\mathbf{M}:\dfrac{\mathcal{G}\mid\Gamma_1,\Gamma_2,\Gamma_3\Rightarrow\Delta_1,\Delta_2,\Delta_3\quad\mathcal{G}\mid\Gamma'_1,\Gamma'_2,\Gamma'_3\Rightarrow\Delta'_1,\Delta'_2,\Delta'_3}{\mathcal{G}\mid\Gamma_1,\Gamma'_1\Rightarrow\Delta_1,\Delta'_1\mid\Gamma_2,\Gamma'_2\Rightarrow\Delta_2,\Delta'_2\mid\Gamma_3,\Gamma'_3\Rightarrow\Delta_3,\Delta'_3}\]

\vspace{.5em}

Logical rules are as in classical logic but with possible ‘side’ sequents.

\vspace{.5em}

\begin{center}
\AxiomC{$\mathcal{G}\mid A_i,\Gamma\Rightarrow\Delta$}
\LeftLabel{$L\wedge$}
\RightLabel{$(i=0,1)$}
\UnaryInfC{$\mathcal{G}\mid A_0\wedge A_1,\Gamma\Rightarrow\Delta$}
\DisplayProof
\quad
\AxiomC{$\mathcal{G}\mid\Gamma\Rightarrow\Delta,A$}
\AxiomC{$\mathcal{G}\mid\Gamma\Rightarrow\Delta,B$}
\LeftLabel{$R\wedge$}
\BinaryInfC{$\mathcal{G}\mid\Gamma\Rightarrow\Delta,A\wedge B$}
\DisplayProof
\end{center}

\vspace{.5em}

\begin{center}
\AxiomC{$\mathcal{G}\mid A,\Gamma\Rightarrow\Delta$}
\AxiomC{$\mathcal{G}\mid B,\Gamma\Rightarrow\Delta$}
\LeftLabel{$L\vee$}
\BinaryInfC{$\mathcal{G}\mid A\vee B,\Gamma\Rightarrow\Delta$}
\DisplayProof
\quad
\AxiomC{$\mathcal{G}\mid\Gamma\Rightarrow\Delta,A_i$}
\LeftLabel{$R\vee$}
\RightLabel{$(i=0,1)$}
\UnaryInfC{$\mathcal{G}\mid\Gamma\Rightarrow\Delta,A_0\vee A_1$}
\DisplayProof
\end{center}

\vspace{.5em}

\begin{center}
	\AxiomC{$\mathcal{G}\mid\Gamma\Rightarrow\Delta,A$}
	\AxiomC{$\mathcal{G}\mid B,\Gamma\Rightarrow\Delta$}
	\LeftLabel{$L\supset$}
	\BinaryInfC{$\mathcal{G}\mid A\supset B,\Gamma\Rightarrow\Delta$}
	\DisplayProof
	\quad
	\AxiomC{$\mathcal{G}\mid A,\Gamma\Rightarrow\Delta,B$}
	\LeftLabel{$R\supset$}
	\UnaryInfC{$\mathcal{G}\mid\Gamma\Rightarrow\Delta,A\supset B$}
	\DisplayProof
\end{center}

\vspace{.5em}

\begin{center}
\AxiomC{$\mathcal{G}\mid\Gamma\Rightarrow\Delta,A$}
\LeftLabel{$L\neg$}
\UnaryInfC{$\mathcal{G}\mid\neg A,\Gamma\Rightarrow\Delta$}
\DisplayProof
\quad
\AxiomC{$\mathcal{G}\mid A,\Gamma\Rightarrow\Delta$}
\LeftLabel{$R\neg$}
\UnaryInfC{$\mathcal{G}\mid\Gamma\Rightarrow\Delta,\neg A$}
\DisplayProof
\end{center}

A $\GLuk$-proof of a~hypersequent $\mathcal{G}$ is a~tree of hypersequents whose leaves are axioms, all other nodes are obtained from their parents via some rule of inference and the root is $\mathcal{G}$.
\end{definition}

In the following theorems, proofs in natural deduction and hypersequent calculus must have a~tree-like form.
\begin{theorem}\label{NDL3toorLuknested} Assume there is an $\mathbf{ND}_{\Luk_3}$ proof of $D$ in $n$ steps. Then there is an~$\Luk_{3_{n\vee}}$-proof of $\Rightarrow D$ in $O(n)$ steps.
\end{theorem}
\begin{proof}
We will prove this theorem in the same way as Theorem~9 from~\cite{BonetBuss1993}.

Observe, first, that if $\Gamma\Rightarrow A$ has an $\Luk_{3_{n\vee}}$ proof in $n$ steps and if $\Gamma'$ is some permutation of formulas in $\Gamma$, then $\Gamma'\Rightarrow A$ also has an $\Luk_{3_{n\vee}}$ proof in $n$ steps, where the first $k$ steps are assumptions from $\Gamma'$.

We will prove by induction on $n$ a~more general statement, namely, that if there is a~tree-like $\mathbf{ND}_{\Luk_3}$ inference of $D$ from $\Delta$ of $n$ steps, then there is an $\Luk_{3_{n\vee}}$ proof of $\Delta'\Rightarrow D$ with $\Delta'$ being some permutation of $\Delta$ without repetitions. Our proof splits into cases depending on how $D$ is inferred.

If $D$ is assumption, then we prove $D\Rightarrow D$ in one step since in this case the only open assumption in the inference of $D$ is $D$ itself. If $D$ is an axiom, then, since $\Luk_{3_{n\vee}}$ is complete, we can prove $\Rightarrow(A\vee\neg A)\vee((B\vee\neg B)\vee(A\supset B))$ in a~constant number of steps (say, $c_0$).

If $D$ is inferred via some rule with one premise, we again can infer it in a~constant number of steps from the premise of that rule. We show this on the case of the $\supset_2$ rule. The last step is \[\dfrac{\neg A}{(B\vee\neg B)\vee(A\supset B)}\] Observe that $\neg A$ is inferred from $\Delta$ in $n-1$ steps, so by the induction hypothesis there is a~proof $\pi$ in $c\cdot(n-1)$ steps of $\Delta\Rightarrow\neg A$. The proof of $\Delta\Rightarrow(B\vee\neg B)\vee(A\supset B)$ is straightforward since we can infer $A\supset B$ from $\neg A$ in a~constant number of steps. From here we need again a~constant number of steps (say, $c_1$) to infer $(B\vee\neg B)\vee(A\supset B)$. Details are left to the reader.
\begin{comment}
\begin{longtable}{l}
$\left\lceil\begin{tabular}{l}
$\Delta$ --- assumptions\\
\vdots\\
$\neg A$ --- from $\pi$\\
\vdots\\
$A\supset B$ --- from $\neg A$\\
\vdots\\
$(B\vee\neg B)\vee(A\supset B)$ --- in $c_1$ steps
\end{tabular}\right.$
\end{longtable}
\end{comment}

Consider now the case of the EFQ rule. Let $A$ be derived from $\Delta_1$ in $n_1$ steps and $\neg A$ from $\Delta_2$ in $n_2$ steps. Since natural deduction is tree-like, $n=n_1+n_2+1$. By the induction hypothesis, there are proofs $\pi_1$ of $\Delta_1\Rightarrow A$ and $\pi_2$ of $\Delta_2\Rightarrow\neg A$ in $c\cdot n_1$ and $c\cdot n_2$ steps. The simulation is straightforward since $(\neg B\supset\neg A)\supset(A\supset B)$ and $\neg A\supset(\neg B\supset\neg A)$ are axioms and we already have $A$ from $\pi_1$ and $\neg A$ from $\pi_2$. From here we need a~constant number of steps (say, $c_2$) to infer $B$. Details are left to the reader.
\begin{comment}
\begin{longtable}{l}
$\left\lceil\begin{tabular}{l}
$\Delta_1\cup\Delta_2$ --- assumptions\\
\vdots\\
$A$ --- from $\pi_1$\\
\vdots\\
$\neg A$ --- from $\pi_2$\\
$(\neg B\supset\neg A)\supset(A\supset B)$ --- axiom\\
$\neg A\supset(\neg B\supset\neg A)$ --- axiom\\
$\neg B\supset\neg A$ --- modus ponens\\
$A\supset B$ --- modus ponens\\
$B$ --- modus ponens
\end{tabular}\right.$
\end{longtable}
\end{comment}

All the other cases of natural deduction rules with two premises are considered in the same way.

Finally, we tackle the case of the $\vee E^{u,v}$ rule. Assume we have the following natural deduction inferences: of $A\vee B$ from $\Delta_1$ in $n_1$ steps, of $C$ from $A$ and $\Delta_2$ in $n_2$ steps and of $C$ from $B$ and $\Delta_3$ in $n_3$ steps. Again, since natural deduction proofs are tree-like, they do not ‘share work’~\cite[p.~703]{BonetBuss1993}, so $n=n_1+n_2+n_3+1$. By the induction hypothesis, we have proofs $\pi_1$, $\pi_2$ and $\pi_3$ of
\[\begin{array}{ccccc}
\Delta_1\Rightarrow A\vee B&\text{and}&\Delta_2*\langle A\rangle\Rightarrow C&\text{and}&\Delta_3*\langle B\rangle\Rightarrow C
\end{array}\] in $c\cdot n_1$, $c\cdot n_2$ and $c\cdot n_3$, respectively.
\begin{longtable}{l}
$\left\lceil\begin{tabular}{l}
$\Delta_1\cup\Delta_2\cup\Delta_3$ --- assumptions\\
\vdots\\
$A\vee B$ --- from $\pi_1$\\
$\left[\begin{tabular}{l}
$A$ --- assumption\\
\vdots\\
$C$ --- from $\pi_2$
\end{tabular}\right.$\\
$\left[\begin{tabular}{l}
$B$ --- assumption\\
\vdots\\
$C$ --- from $\pi_3$
\end{tabular}\right.$\\
$C$ --- $\vee_{ne}$
\end{tabular}\right.$
\end{longtable}

The proof has $c\cdot n_1+c\cdot n_2+c\cdot n_3+1$ steps.

The result follows if we take $c\geqslant c_0,c_1,c_2$.
\end{proof}

As an immediate corollary, we obtain the following proposition.
\begin{corollary}\label{NDL3toHLuk}
Assume there is a~natural deduction proof of $A$ in $n$ steps. Then there is an {$\HLuk$} proof of $A$ in $O(n\alpha(n))$ steps.
\end{corollary}

In contrast to sequent calculi like those presented in~\cite{Gentzen1935-1,Gentzen1935-2,Takeuti1987,Buss1998HPT,TroelstraSchwichtenberg2000}, the total number of formulas in a~hypersequent inferred on the step $n$ is not bounded by $n+1$. Hence the usual technique from~\cite{BonetBuss1993} to provide an upper bound on speed-up w.r.t.\ number of steps in the proof does not work. That is why our next theorem provides an upper bound w.r.t.\ the number of symbols, i.e., the \textit{length} of the proof.
\begin{convention}
We will further denote the length of formula $A$, i.e., the number of symbols it contains, as $|A|$.
\end{convention}
\begin{theorem}\label{GLuktoorLuknested} Assume there is a~{$\GLuk$} proof of $$A^1_1,\ldots,A^1_{k_1}\Rightarrow B^1_1,\ldots,B^1_{l_1}\mid\ldots\mid A^m_1,\ldots,A^m_{k_m}\Rightarrow B^m_1,\ldots,B^m_{l_m}$$ of length $n$ which contains $N$ steps. Then there is an~{$\HLuk$} proof of $$\bigvee\limits^{m}_{i=1} A^i_1\supset(\ldots\supset(A^i_{k_i}\supset(\neg B^i_1\supset(\ldots(\neg B^i_{l_i-1}\supset B^i_{l_i})))))$$ containing $O(n^3)$ symbols, where disjunction is associated to the right.
\end{theorem}
\begin{proof}
We first translate every hypersequent \[\mathcal{G}_i=C^1_1,\ldots,C^1_{k_1}\Rightarrow D^1_1,\ldots,D^1_{l_1}\mid\ldots\mid C^m_1,\ldots,C^m_{k_m}\Rightarrow D^m_1,\ldots,D^m_{l_m}\] from the $\GLuk$ proof into logically equivalent formula
\[\bigvee\limits^{m}_{i=1} C^i_1\supset(\ldots\supset(C^i_{k_i}\supset(\neg D^i_1\supset(\ldots(\neg D^i_{l_i-1}\supset D^i_{l_i})))))\]
with disjunction being associated to the right. Let us denote these new formulas as $\phi_{\mathcal{G}_i}$. Observe that if the length of $\mathcal{G}_i$ is $|\mathcal{G}_i|$, then each of these formulas contains $O(|\mathcal{G}_i|)$ symbols, hence $n=\sum\limits^{N}_{i=1}|\mathcal{G}_i|$. The translation gives a~sequence of formulas
\[\phi_{\mathcal{G}_1},\ldots,\phi_{\mathcal{G}_n}\]
which contains $O(n)$ symbols in total.

It now remains to show that we can fill in every gap in $O(n^2)$ symbols. More precisely, we will show that the following statement holds.

\textit{If $\mathcal{G}_n=\mathcal{H}\mid\Gamma\Rightarrow\Delta$, with $\Gamma$ and $\Delta$ containing $k$ formulas in total, is inferred from $\mathcal{G'}=\mathcal{H}\mid\Gamma'\Rightarrow\Delta'$ and $\mathcal{G''}=\mathcal{H}\mid\Gamma''\Rightarrow\Delta''$ via some rule, then we need no more than $O(k)$ steps, each of which contains $O(\max(|\mathcal{H}_1|,|\mathcal{H}_2|,|\mathcal{G}_n|))$ symbols, to fill in the gap.}

Since $k\leqslant\max(|\mathcal{H}_1|,|\mathcal{H}_2|,|\mathcal{G}_n|)$, it suffices to prove the italicised statement. Our proof splits depending on how $\mathcal{G}_n$ is inferred. We will consider only the most instructive cases.
\begin{case}
Axiom.
\end{case}
If $\mathcal{G}_n$ is an axiom $A\Rightarrow A$, then $\phi_{\mathcal{G}_n}$ is $A\supset A$. Hence it has an $\HLuk$ proof in a~constant number of steps. Since each step has $O(|A|)$ symbols, the proof of $A\supset A$ is of length $O(|A|)$.
\begin{case}
External structural rules.
\end{case}
We consider \textbf{EC} and \textbf{EP} rules.

If $\mathcal{G}_n$ is inferred via \textbf{EC} from $\mathcal{G}_n\mid\mathcal{G}_n$, then $\phi_{\mathcal{G}_n\mid\mathcal{G}_n}=\phi_{\mathcal{G}_n}\vee\phi_{\mathcal{G}_n}$ from where there is an $\HLuk$ proof of $\phi_{\mathcal{G}_n}$ in a~constant number of steps and $O(|\phi_{\mathcal{G}_n}|)$ symbols.

Now, if
$$\mathcal{G}_n=\Gamma_1\Rightarrow\Delta_1\mid\ldots\mid\Gamma_k\Rightarrow\Delta_k\mid\Gamma_l\Rightarrow\Delta_l\mid\ldots\mid\Gamma_m\Rightarrow\Delta_m$$
is inferred via \textbf{EP} from
\[\mathcal{G}=\Gamma_1\Rightarrow\Delta_1\mid\ldots\mid\Gamma_l\Rightarrow\Delta_l\mid\Gamma_k\Rightarrow\Delta_k\mid\ldots\mid\Gamma_m\Rightarrow\Delta_m\]
then we proceed as follows. First, let us denote every translation of the sequent $\Gamma_i\Rightarrow\Delta_i$ as $\sigma_i$. Hence $\phi_{\mathcal{G}_n}=\sigma_1\vee\ldots\vee\sigma_k\vee\sigma_l\vee\ldots\vee\sigma_m$ and $\phi_{\mathcal{G}_{n-1}}=\sigma_1\vee\ldots\vee\sigma_l\vee\sigma_k\vee\ldots\vee\sigma_m$ with both disjunctions being associated to the right. Observe also that $m\leqslant|\phi_{\mathcal{G}_n}|$ and $|\phi_{\mathcal{G}_{n-1}}|=|\phi_{\mathcal{G}_n}|$. It can be proved by induction on $m$ that there is an $\HLuk$ inference $\phi_{\mathcal{G}_n}$ from $\phi_{\mathcal{G}_{n-1}}$ in $O(m)$ steps. But since each step contains $O(|\phi_{\mathcal{G}_n}|)$ symbols, the proof will be of length $O(|\phi_{\mathcal{G}_n}|^2)$.
\begin{case}
Internal structural rules.
\end{case}
We will handle \textbf{IP} in the succedent of the sequent and \textbf{M}.

If $\mathcal{G}_n$ is inferred via \textbf{IP} from $\mathcal{G}$, then
$$\phi_{\mathcal{G}}=\phi_{\mathcal{H}}\vee(E_1\supset(\ldots\supset(E_i\supset(E_j\supset(\ldots\supset(E_{l-1}\supset E_{l}))))))$$
and either
$$\phi_{\mathcal{G}_{n}}=\phi_{\mathcal{H}}\vee(E_1\supset(\ldots\supset(E_j\supset(E_i\supset(\ldots\supset(E_{l-1}\supset E_{l}))))))$$
or
$$\phi_{\mathcal{G}_{n}}=\phi_{\mathcal{H}}\vee(E_1\supset(\ldots\supset(E_i\supset(E_j\supset(\ldots\supset(\neg E_{l}\supset\neg E_{l-1}))))))$$
Observe that $|\phi_{\mathcal{G}_{n}}|=|\phi_{\mathcal{G}}|+c$ with $c=0,2$.

In the first case we can infer $\phi_{\mathcal{G}_{n}}$ from $\phi_{\mathcal{G}_{n'}}$ as follows. First, in one step we prove $(\neg E_l\supset\neg E_{l-1})\supset(E_{l-1}\supset E_l)$. This gives us $O(|\phi_{\mathcal{G}_{n}}|)$ symbols. From here it takes $O(l)$ steps to infer $\phi_{\mathcal{G}_{n'}}\supset\phi_{\mathcal{G}_{n}}$,
where each step will contain $O(|\phi_{\mathcal{G}_{n}}|)$ symbols but since $l<|\phi_{\mathcal{G}_{n}}|$, there will be $O(|\phi_{\mathcal{G}_{n}}|^2)$ symbols in total. After that we implement modus ponens and infer $\phi_{\mathcal{G}_{n}}$ with $O(|\phi_{\mathcal{G}_{n}}|)$ symbols.

In the second case we first infer
$$(E_i\supset(E_j\supset(\ldots\supset(E_{l-1}\supset E_l))))\supset(E_j\supset(E_i\supset(\ldots\supset(E_{l-1}\supset E_l))))$$
in a~constant number of steps and $O(|\phi_{\mathcal{G}_{n}}|)$ symbols. From here there is an inference in $O(l-i)$ steps of $\phi_{\mathcal{G}_{n'}}\supset\phi_{\mathcal{G}_{n}}$, where each step contains $O(|\phi_{\mathcal{G}_{n}}|)$ symbols. The inference, hence, contains $O(|\phi_{\mathcal{G}_{n}}|^2)$ symbols. After that we implement modus ponens and infer $\phi_{\mathcal{G}_{n}}$ with $O(|\phi_{\mathcal{G}_{n}}|)$ symbols.

Assume
$$\mathcal{H}_1=\mathcal{H}\mid A_1,\ldots,A_{i_1},B_1,\ldots,B_{i_2},C_1,\ldots,C_{i_3}\Rightarrow D_1,\ldots D_{j_1},E_1,\ldots,E_{j_2},F_1,\ldots,F_{j_3}$$
$$\mathcal{H}_2=\mathcal{H}\mid A'_1,\ldots,A'_{i'_1},B'_1,\ldots,B_{i'_2},C'_1,\ldots,C_{i'_3}\Rightarrow D'_1,\ldots D'_{j'_1},E_1,\ldots,E'_{j'_2},F'_1,\ldots,F'_{j'_3}$$ 

Then if $\mathcal{G}_n$ is inferred from $\mathcal{H}_1$ and $\mathcal{H}_2$ via \textbf{M},
$$\mathcal{G}_n=\mathcal{H}\mid \Gamma_1\Rightarrow\Delta_1\mid\Gamma_2\Rightarrow\Delta_2\mid\Gamma_3\Rightarrow\Delta_3$$
with
$$
\begin{array}{rl}
\Gamma_1=\langle A_1,\ldots,A_{i_1},A'_1,\ldots,A'_{i'_1}\rangle&\Delta_1=\langle D_1,\ldots,D_{j_1},D'_1,\ldots,D'_{j'_1}\rangle\\
\Gamma_2=\langle B_1,\ldots,B_{i_2},B'_1,\ldots,B'_{i'_2}\rangle&\Delta_2=\langle E_1,\ldots,E_{j_2},E'_1,\ldots,E'_{j'_2}\rangle\\
\Gamma_3=\langle C_1,\ldots,C_{i_3},C'_1,\ldots,C'_{i'_3}\rangle&\Delta_3=\langle F_1,\ldots,F_{j_3},F'_1,\ldots,F'_{j'_3}\rangle
\end{array}
$$

Hence
$$\phi_{\mathcal{H}_1}=\phi_{\mathcal{H}}\vee(A_1\supset(\ldots\supset(C_{i_3}\supset\neg D_1(\ldots\supset(\neg F_{j_3-1}\supset F_{j_3})))))$$
$$\phi_{\mathcal{H}_2}=\phi_{\mathcal{H}}\vee(A'_1\supset(\ldots\supset(C'_{i'_3}\supset\neg D'_1(\ldots\supset(\neg F'_{j'_3-1}\supset F'_{j'_3})))))$$
and
\begin{eqnarray}
\phi_{\mathcal{G}_{n}}=\phi_{\mathcal{H}}\vee(\underbrace{A_1\supset(\ldots\supset(A'_{i'_1}}_{\text{from }\Gamma_1}\supset\underbrace{(\neg D_1\supset(\ldots\supset(\neg D'_{j'_1-1}\supset D'_{j'_1})))))}_{\text{from }\Delta_1})\nonumber\\
\vee(\underbrace{B_1\supset(\ldots\supset (B'_{i'_2}}_{\text{from }\Gamma_2}\supset\underbrace{(\neg E_1\supset(\ldots\supset(\neg E'_{j'_2-1}\supset E'_{j'_2})))))}_{\text{from }\Delta_2})\nonumber\\
\vee(\underbrace{C_1\supset(\ldots\supset(C'_{i'_3}}_{\text{from }\Gamma_3}\supset\underbrace{(\neg F_1\supset(\ldots\supset(\neg F'_{j'_3-1}\supset F'_{j'_3})))))}_{\text{from }\Delta_3})\nonumber
\end{eqnarray}
where disjunctions are associated to the right.

Now let $k$ be the number of formulas in $\Gamma_1\Rightarrow\Delta_1\mid\Gamma_2\Rightarrow\Delta_2\mid\Gamma_3\Rightarrow\Delta_3$. It can be shown by induction on $k$ that $\phi_{\mathcal{G}_{n}}$ can be inferred from $\phi_{\mathcal{H}_1}$ and $\phi_{\mathcal{H}_2}$ in $O(k)$ steps with each step containing $O(|\phi_{\mathcal{G}_{n}}|)$ symbols.
\begin{case}
Left-hand side logical rules.
\end{case}
Assume $\mathcal{G}_n$ is inferred from $\mathcal{H}_1$ and $\mathcal{H}_2$ via $L\!\supset$ as follows.
\[\dfrac{\mathcal{H}\mid A_1,\ldots,A_k\Rightarrow B_1,\ldots,B_l,A\quad\mathcal{H}\mid B,A_1,\ldots,A_{k}\Rightarrow B_1,\ldots,B_{l}}{\mathcal{H}\mid A\supset B,A_1,\ldots,A_k\Rightarrow B_1,\ldots,B_l}\]

Here
\[\phi_{\mathcal{H}_1}=\phi_{\mathcal{H}}\vee(A_1\supset(\ldots\supset(A_{k}\supset(\neg B_1\supset(\ldots\supset(\neg B_l\supset A))))))\]
\[\phi_{\mathcal{H}_2}=\phi_{\mathcal{H}}\vee(B\supset(A_1\supset(\ldots\supset(A_{k}\supset(\neg B_1\supset(\neg B_{l-1}\supset B_l))))))\]
and
\[\phi_{\mathcal{G}_{n}}=\phi_{\mathcal{H}}\vee((A\supset B)\supset(A_1\supset(\ldots\supset(A_{k}\supset(\neg B_1\supset(\neg B_{l-1}\supset B_l))))))\]

The simulation goes as follows. Observe that $|\phi_{\mathcal{G}_{n}}|=|\phi_{\mathcal{H}_1}|+|B|+c_1=|\phi_{\mathcal{H}_2}|+|A|+c_2$ for some constants $c_1$ and $c_2$. We first infer
\[\phi'_{\mathcal{H}_1}=\phi_{\mathcal{H}}\vee(\neg A\supset(A_1\supset(\ldots\supset(A_{k}\supset(\neg B_1\supset(\neg B_{l-1}\supset B_l))))))\]
from $\phi_{\mathcal{H}_1}$ in $O(k+l)$ steps each of which contains $O(|\phi_{\mathcal{H}_1}|)$ symbols.

Now we need a~constant number of steps with $O(|\phi_{\mathcal{G}_{n}}|)$ symbols each to infer $\phi_{\mathcal{G}_{n}}$ from $\phi'_{\mathcal{H}_1}$ and~$\phi_{\mathcal{H}_2}$.
\begin{case}
Right-hand side logical rules.
\end{case}
Assume $\mathcal{G}_n$ is inferred from $\mathcal{G}$ $R\!\supset$ as follows.
\[\dfrac{\mathcal{H}\mid A,A_1,\ldots,A_k\Rightarrow B_1,\ldots,B_l,B}{\mathcal{H}\mid A_1,\ldots,A_k\Rightarrow B_1,\ldots,B_l,A\supset B}\]

Here
\[\phi_{\mathcal{G}}=\phi_{\mathcal{H}}\vee(A\supset(A_1\supset(\ldots\supset(A_k\supset(\neg B_1\supset(\ldots\supset(\neg B_l\supset B)))))))\]
\[\phi_{\mathcal{G}_n}=\phi_{\mathcal{H}}\vee(A_1\supset(\ldots\supset(A_k\supset(\neg B_1\supset(\ldots\supset(\neg B_l\supset(A\supset B)))))))\]

Observe that $|\phi_{\mathcal{G}}|=|\phi_{\mathcal{G}_{n}}|$. Let $m=k+l$. Then it can be shown by induction on $m$ that there is a~derivation of $\phi_{\mathcal{G}_{n}}$ from $\phi_{\mathcal{G}}$ in $O(m)$ steps each of which contains $O(|\phi_{\mathcal{G}_{n}}|)$ symbols.
\end{proof}

Recall that Theorems~\ref{HLuktoorLuk}, \ref{HLuktoorLuknested} and~\ref{NDL3toorLuknested} are obtained by filling in the gaps that appeared after translations of inferences from one proof system to the other. By proof inspection, we observe that the number of symbols in each step increases only by a~constant factor in comparison to the original inference. Hence, the following corollaries.
\begin{corollary}\label{HLuktoorLuksymbols}
Assume there is an $\Luk_{3_\vee}$ proof of $A$ containing $m$ symbols\footnote{We count the total number of symbols in sequents in Corollaries~\ref{HLuktoorLuksymbols}, \ref{HLuktoorLuknestedsymbols}, \ref{HkLuktoorLuksymbols} and~\ref{HkLuktoorLuknestedsymbols}.} and $k$~steps. Then there is an {$\HLuk$} proof of $A$ containing $O(m\cdot k\log k)$ symbols.
\end{corollary}
\begin{corollary}\label{HLuktoorLuknestedsymbols}
Assume there is an $\Luk_{3_{n\vee}}$ proof of $A$ containing $m$ symbols and $k$~steps. Then there is an {$\HLuk$} proof of $A$ containing $O(m\cdot\alpha(k))$ symbols.
\end{corollary}
\begin{corollary}
Assume there is an $\mathbf{ND}_{\Luk_3}$ proof of $A$ containing $m$ symbols and $k$ steps. Then there is an $\Luk_{3_{n\vee}}$ proof of $\Rightarrow A$ containing $O(m)$ symbols.
\end{corollary}
\section[Generalisation to $\Luk_k$]{Generalisation to finite-valued \L{}ukasiewicz logics\label{Lk}}
\subsection{Semantics and proof systems}
For any finite $k$ we define semantics of propositional $k$-valued \L{}ukasiewicz logic $\Luk_k$ as follows.
\begin{enumerate}
\item A valuation $v$ maps variables to $\left\{1,\frac{k-2}{k-1},\ldots,\frac{1}{k-1},0\right\}$.
\item $v(\neg A)=1-v(A)$.
\item $v(A\wedge B)=\min(v(A),v(B))$.
\item $v(A\vee B)=\max(v(A),v(B))$.
\item $v(A\supset B)=\min(1,1-v(A)+v(B))$.
\end{enumerate}

The notions of a~valid formula and entailment are the same as in $\Luk_3$.

It seems that for most finite-valued \L{}ukasiewicz logics the only known proof systems are Frege systems (cf., for instance~\cite{Tuziak1988} and~\cite{Karpenko2006}) which we will call $\HLuk_k$. That is why we will consider only the proof systems obtained by augmenting Frege system with either ‘nested’ or ‘general’ version of the disjunction elimination rule. We will denote these systems as $\Luk_{k_{n\vee}}$ and $\Luk_{k_\vee}$, respectively.

The notions of a~proof for $\Luk_{k_{n\vee}}$ and $\Luk_{k_\vee}$ are easily adapted from those for $\Luk_{3_{n\vee}}$ and $\Luk_{3_\vee}$ (cf.~Definitions~\ref{orLuk} and~\ref{orLuknested}).
\subsection[Simulation of $\vee_e$ and $\vee_{ne}$ in $\Luk_k$]{Simulations of the general and nested disjunction elimination rules for~$\Luk_k$}
Recall quickly that the following statement holds (here $\vDash_{\Luk_k}$ is the entailment relation of $\Luk_k$).
\[A\vDash_{\Luk_k}\text{ iff }\vDash_{\Luk_k}\underbrace{A\supset(\ldots\supset(A}_{k-1\text{ times}}\supset B))\text{ iff }\vDash_{\Luk_k}\neg\underbrace{(A\supset(\ldots\supset(A\supset\neg A)))}_{k-1\text{ times}}\supset B\]
\begin{lemma}\label{HkLukextraction}
Suppose that $\mathfrak{B}$ is an arbitrarily associated conjunction $A_1\wedge\ldots\wedge A_m$ and $\mathfrak{C}$ is some arbitrarily associated conjunction $A_{i_1}\wedge\ldots\wedge A_{i_n}$, where $i_{1}$, \ldots, $i_{n}$ is any sequence from $\{1,\ldots,m\}$.  Then there is an {$\HLuk_k$} proof of $\mathfrak{B}\supset\mathfrak{C}$ in $O(n\log m)$ steps for any finite $k$.
\end{lemma}
\begin{proof}
Analogously to Lemma~\ref{HLukextraction}.
\end{proof}
Now we can prove two theorems that shed some light on upper bounds on speed-ups for arbitrary finite-valued \L{}ukasiewicz logics.
\begin{theorem}\label{HkLuktoorkLuk}
For any finite $k$ if there is an $\Luk_{k_\vee}$ proof of $D$ in $n$ steps, then there is an {$\HLuk_k$} proof of $D$ in $O(n^2\log n)$ steps.
\end{theorem}
\begin{proof}
Let us denote the conjunction of $m$ $\neg\underbrace{(A_i\supset(\ldots\supset(A_i\supset\neg A_i)))}_{k-1\text{ times}}$ formulas associated and ordered arbitrarily as $\bigotimes\limits^{m}_{i=1}A$.

%It suffices to show that if $\{C_1,\ldots,C_m\}\Rightarrow D$ has an $\Luk_{k_\vee}$ proof of $n$ steps, then $\bigotimes\limits^{m}_{i=1}C_i\supset D$ has $\HLuk_k$ proof in $O(n^2\log n)$ steps.

The proof is done in the same way as in Theorem~\ref{HLuktoorLuk} with the exception that we will use Lemma~\ref{HkLukextraction} instead of Lemma~\ref{HLukextraction} and all occurrences of $\neg(A\supset\neg A)$ formulas should be substituted for the formulas of the form $\neg\underbrace{(A\supset(\ldots\supset(A\supset\neg A)))}_{k-1\text{ times}}$.
\end{proof}
\begin{theorem}\label{HkLuktoorkLuknested}
For any finite $k$ the following holds.

(1) Assume there is an $\Luk_{k_{n\vee}}$ proof of $A$ in $n$ steps. Then there is an {$\HLuk_k$} proof of $A$ in $O(n\cdot\alpha(n))$ steps with $\alpha$ being inverse Ackermann function as defined in~\cite{BonetBuss1993} and~\cite{BonetBuss1995}.

(2) Assume there is an $\Luk_{k_{n\vee}}$ proof of $A$ in $n$ steps with assumptions been opened $m$ times. Then there is an {$\HLuk_k$} proof of $A$ in $O(n\cdot m\log^{(*i)}m)$ steps.
\end{theorem}
\begin{proof}
Let $\Gamma=\langle A_1,\ldots,A_l\rangle$. We denote
$$\neg\underbrace{(A_1\supset(\ldots\supset(A_1\supset\neg A_1)))}_{k-1\text{ times}}\wedge\ldots\wedge\neg\underbrace{(A_l\supset(\ldots\supset(A_l\supset\neg A_l)))}_{k-1\text{ times}}$$
where conjunction is associated to the right with $\bigotimes\Gamma$.

The proof is conducted in the same way as Theorem~\ref{HLuktoorLuk} with the exception that all occurrences of $\neg(A\supset\neg A)$ should be substituted for $\neg\underbrace{(A\supset(\ldots\supset(A\supset\neg A)))}_{k-1\text{ times}}$ formulas.
\end{proof}

We conclude this section with the analogues of Corollaries~\ref{HLuktoorLuksymbols} and~\ref{HLuktoorLuknestedsymbols}.
\begin{corollary}\label{HkLuktoorLuksymbols}
For any finite $k$ the following holds: assume there is an $\Luk_{k_\vee}$ proof of $A$ containing $m$ symbols and $n$~steps. Then there is an {$\HLuk$} proof of $A$ containing $O(m\cdot n\log n)$ symbols.
\end{corollary}
\begin{corollary}\label{HkLuktoorLuknestedsymbols}
For any finite $k$ the following holds: assume there is an $\Luk_{k_{n\vee}}$ proof of $A$ containing $m$ symbols and $n$~steps. Then there is an {$\HLuk$} proof of $A$ containing $O(m\cdot\alpha(n))$ symbols.
\end{corollary}
\section[Conclusion]{Conclusion and future work\label{conclusion}}
In this paper we have shown that it is possible in the case of finite-valued \L{}uksaiewicz logics to provide upper bounds on speed-ups for ‘nested’ and ‘general’ systems either the same as in classical case (Theorems~\ref{HLuktoorLuknested} and~\ref{HkLuktoorkLuknested}) or only slightly worse (Theorems~\ref{HLuktoorLuk} and~\ref{HkLuktoorkLuk}). To obtain this result, we used the disjunction elimination rule instead of the deduction rule. Moreover, the upper bound on speed-up of natural deduction over Hilbert-style calculi (Theorem~\ref{NDL3toorLuknested} and Corollary~\ref{NDL3toHLuk}) for $\Luk_3$ is also the same as the upper bound on speed-up of natural deduction over classical Frege systems. We have also established polynomial simulation of $\GLuk$ by $\HLuk$ w.r.t.\ lengths of proofs (Theorem~\ref{GLuktoorLuknested}).

Several questions, however, remain open.
\begin{enumerate}
\item We relied heavily on the fact that implication in $\Luk_k$ is transitive.  What will happen in logics whose implications are not transitive? Is it possible to provide a~characterisation of properties of implication necessary and sufficient for the same (or at least polynomial) upper bounds on speed-ups as in classical logic?
\item We restricted ourselves to finite-valued logics. This allowed us to construct formulas consisting of $A$ and $B$ which are true iff $A\vDash_{\Luk_k}B$. Is it possible to generalise the procedures and retain speed-ups given here for the case of infinite-valued logics, in particular, \L{}ukasiewicz logic and product logic which do not have a~uniform kind of formulas consisting of $A$ and $B$ which are valid iff $A$ entails $B$?
\item Is it possible to improve any of non-linear upper bounds?
\end{enumerate}
\bibliographystyle{plain}
\bibliography{references.bib}
\end{document}